\documentclass[11pt,letterpaper]{amsart}
\usepackage{amsmath}
\usepackage{amssymb}
\usepackage{amsthm}
\usepackage{hyperref}
\usepackage[noabbrev,capitalize]{cleveref}
\usepackage{mathrsfs}
\usepackage{mathtools}
\usepackage{slashed}
\usepackage{tikz-cd}
\usepackage[shortlabels]{enumitem}

\setenumerate{label=(\roman*)}

\DeclareMathOperator{\Ric}{Ric}
\DeclareMathOperator{\biRic}{biRic}

\def\sideremark#1{\ifvmode\leavevmode\fi\vadjust{\vbox to0pt{\vss
 \hbox to 0pt{\hskip\hsize\hskip1em
 \vbox{\hsize3cm\tiny\raggedright\pretolerance10000
 \noindent #1\hfill}\hss}\vbox to8pt{\vfil}\vss}}}

\newtheorem{theorem}{Theorem}[section]
\newtheorem{proposition}[theorem]{Proposition}
\newtheorem{lemma}[theorem]{Lemma}
\newtheorem{corollary}[theorem]{Corollary}

\theoremstyle{definition}
\newtheorem{definition}[theorem]{Definition}

\theoremstyle{remark}
\newtheorem{remark}[theorem]{Remark}

\numberwithin{equation}{section}
 \usepackage{ stmaryrd }
\usepackage{import}

\allowdisplaybreaks
\makeatletter
\@namedef{subjclassname@2020}{\textup{2020} Mathematics Subject Classification}
\makeatother

\begin{document}

\title[Spectral nonnegative Ricci curvature]{A splitting theorem for manifolds with  nonnegative spectral Ricci curvature and mean convex boundary}

\author{Han Hong}
\address{Department of Mathematics and statistics \\ Beijing Jiaotong University \\ Beijing \\ China, 100044}
\email{hanhong@bjtu.edu.cn}

\author{Gaoming Wang}
\address{Beijing Institute of Mathematical Sciences and Applications \\ Beijing \\ China, 101408}
\address{Yau Mathematical Sciences Center, Tsinghua University, Beijing, China, 100084}
\email{wanggaoming@bimsa.cn}

\begin{abstract}
We prove a splitting theorem for a smooth noncompact manifold with (possibly noncompact) boundary. We show that if a noncompact manifold of dimension $n\geq 2$ has $\lambda_1(-\alpha\Delta+\operatorname{Ric})\geq 0$ for some $\alpha<\frac{4}{n-1}$ and mean convex boundary, then it is either isometric to $\Sigma\times \mathbb{R}_{\geq 0}$ for a closed manifold $\Sigma$ with nonnegative Ricci curvature or it has no interior ends.
\end{abstract}
\maketitle

\section{Introduction}
The classical Cheeger-Gromoll splitting theorem \cite{cheegergromoll} states that a complete noncompact manifold with at least two ends and nonnegative Ricci curvature must be isometric to $ \Sigma\times \mathbb{R}_+$ where $\Sigma$ is a closed manifold with nonnegative Ricci curvature. This result has been extended to manifolds with compact boundary by Kasue \cite{kasue} and Croke-Kleiner \cite{croke-bruce}. More specifically, a noncompact manifold with a compact mean convex boundary and nonnegative Ricci curvature is isometric to $\mathbb{R}_+\times \Sigma$. The compactness of the boundary is essential, as illustrated by the example of the catenoid, which bounds a domain in $\mathbb{R}^n$.

Recently, Antonelli-Pozzetta-Xu \cite{antonelli-pozzetta-xu} extended the Cheeger-Gromoll splitting theorem to manifolds with nonnegative Ricci curvature in the spectral sense (to be defined below). Catino-Mari-Mastrolia-Roncoroni \cite{catino-Mari-Mastrolia-Roncoroni} proved the same theorem using a different approach. Many classical and important theorems in Riemannian geometry have been generalized under this framework, for example, \cite{antonelli-xu}. In particular, these extensions have played a significant role in the resolution of the stable Bernstein problem \cite{chodoshliR4anisotropic,chodosh-li-stryker,mazet}. The goal of this paper is to establish a spectral splitting theorem for manifolds with boundary.

Let $(M,g)$ be a noncompact $n$-dimensional Riemannian manifold with boundary $\partial M$. Denote the outward unit normal vector of $\partial M$ by $\eta$. One can define the Lipschitz function $\operatorname{Ric}: M \rightarrow \mathbb{R}$ by
\[
\operatorname{Ric}(x) = \inf\{\operatorname{Ric}(v, v) : g(v, v) = 1 \text{ and } v \in T_xM\}.
\]
For a constant $\alpha \geq 0$, one says that $M$ has nonnegative $\alpha$-Ricci curvature in the spectral sense if $\lambda_1(-\alpha \Delta + \operatorname{Ric}) \geq 0$.

Equivalently, this can be expressed as
\[
\int_M \alpha |\nabla \varphi|^2 + \operatorname{Ric} \cdot \varphi^2 \geq 0
\]
for any compactly supported function $\varphi$, or there exists a positive function $u \in C^{2,\beta}(M)$ satisfying
\begin{equation}\label{equation 1 of u}
    -\alpha \Delta u + \operatorname{Ric} \cdot u = 0 \quad \text{on } M,
\end{equation}
\begin{equation}\label{equation 2 of u}
    \frac{\partial u}{\partial \eta} = 0 \quad \text{on } \partial M.
\end{equation}
This equivalence follows from arguments similar to those of Fischer-Colbrie and Schoen \cite{Fischer-Colbrie-Schoen-The-structure-of-complete-stable}.

The boundary $\partial M$ consists of components that may be compact or noncompact. As in the case of geodesically complete manifolds, one can define the notion of ends. 
Fix a point $p \in \partial M$, and choose a large subset containing $p$, e.g., $B_R(p)$ with large $R$, such that the complement $M \setminus B_R(p)$ has finitely many connected unbounded components. These components are called \textit{ends} with respect to $B_R(p)$.
There are two types of ends: some have a boundary portion belonging to $\partial M$, while others only have a boundary belonging to $\partial B_R(p)$. The former are called \textit{boundary ends}, and the latter are called \textit{interior ends}. In dimension two, these concepts have been introduced to study noncompact capillary surfaces by the author \cite{hongcrelle}. An example is that one can imagine attaching very tiny half-cylinders to a long strip at integer interior points. See also \cite{wuyujie,hong24} for some discussions in higher dimensions.

A manifold is said to have mean convex boundary if the mean curvature of the boundary with respect to outward normals is nonnegative, i.e., $h(\cdot,\cdot)=g(\nabla_{\cdot}\eta,\cdot)$ is nonnegative. The main result of this paper is the following.
 
 \begin{theorem}\label{maintheorem}
    Let $(M, g)$ be a smooth noncompact $n$-dimensional Riemannian manifold with nonnegative $\alpha$-Ricci curvature in the spectral sense and mean convex boundary. Assume that $\alpha < \frac{4}{n-1}$. Then, either
    \begin{itemize}
        \item[1.] $M$ is isometric to a product space $\Sigma \times \mathbb{R}_{\geq 0}$ with a compact manifold $\Sigma$ satisfying $\operatorname{Ric}_\Sigma \geq 0$, or
        \item[2.] $M$ has no interior ends.
    \end{itemize}
\end{theorem}

Notice that the compactness of $\partial M$ is not assumed a priori but rather emerges as a consequence of the splitting result. Regarding to the sharpness of the assumptions, we make two comments here. Both examples below are simple truncations of examples provided in \cite{antonelli-pozzetta-xu}. 

(1) The range for $\alpha$ is sharp. Otherwise, consider the manifold $((-\infty,0]\times \mathbb{T}^{n-1},dt^2+e^{2t}g_{\mathbb{T}^n})$. This is a quotient of the domain $\{x_n\geq 1\}$ in the upper half-space model of hyperbolic space. The mean curvature of the boundary is $n-1$ and $\Ric=-(n-1)$. For $\alpha\geq \frac{4}{n-1}$, take $u=e^{-\frac{n-1}{2}t}$ and calculate
\[-\alpha \Delta u+\Ric \cdot u\geq0.\]
This manifold has one interior end, but it does not split isometrically.

(2) The assumption of at least one interior end in item 1 of Theorem \ref{maintheorem} guarantees the existence of a geodesic ray. However, the existence of a geodesic ray does not lead to the splitting.  Consider a sufficiently small compact perturbation of $\mathbb{R}^n_+$ as in \cite[Remark 4.2]{antonelli-pozzetta-xu}. The new manifold satisfies nonnegative $\alpha$-Ricci curvature in the spectral sense, and the boundary has nonnegative mean curvature. It contains geodesic rays, but it does not split isometrically. Note that this manifold only has one boundary end.

The proof of Theorem \ref{maintheorem} relies on two key techniques: the second variation of the weighted length functional and the curve-capture argument.
First, we analyze the second variation of the weighted length functional. By computing the second variation, we establish that the associated weighted function remains constant along any weighted minimizing line or ray. Additionally, it implies that the Ricci curvature is nonnegative along such curves.

Next, we employ the curve-capture technique to construct a weighted minimizing ray or line passing through any given point $p\in M$. This construction is analogous to the prescribed surface method used to resolve the Milnor conjecture in dimension 3 \cite{liugang}. Further applications of this technique in splitting theorems can be found in \cite{anderson-rodriguez,Carlotto-Chodosh-Eichmair-PMT,chodoshsplittingtheorem,zhu1,antonelli-pozzetta-xu,hongwang-splitting}.
In particular, our approach provides a distinct strategy compared to those in \cite{antonelli-pozzetta-xu,catino-Mari-Mastrolia-Roncoroni}.
Finally, we conclude that the Ricci curvature is non-negative on $M$ and finish the proof by the result due to Sakurai \cite{sakurai2017}.

\vskip.1cm
In the following, we give two corollaries of the main result. The first one is a generalization of \cite[Theorem 2, $\delta=0$ case]{croke-bruce}.

\begin{corollary}
    Let \(M\) be a noncompact $n$-dimensional manifold with compact, mean convex boundary. Suppose it supports nonnegative \(\alpha\)-Ricci curvature in the spectral sense for \(\alpha < \frac{4}{n-1}\),  then $M$ is isometric to a product space $\Sigma \times \mathbb{R}_{\geq 0}$ with a compact manifold $\Sigma$ satisfying $\operatorname{Ric}_\Sigma \geq 0$.
\end{corollary}

We now show the second corollary regarding the stable CMC hypersurface in a Riemannian manifold. Recall that, given two orthonormal vectors $x,y\in T_p N$,  the biRic curvature operator is defined as
\[\operatorname{biRic}(x,y):=\operatorname{Ric}(x,x)+\operatorname{Ric}(y,y)-R(x,y,x,y).\]
One can define the biRic curvature as
\[\operatorname{biRic}(p)=\min\{\operatorname{biRic}(x,y):\ x, y\ \text{are orthonormal vectors in}\ T_pN\}.\]
The biRic curvature was first raised by Shen-Ye \cite{shenyingyerugang} to study stable minimal hypersurfaces. It and its variants have played an important role in the resolution of the stable Bernstein problem in $\mathbb{R}^{n+1}$ (see \cite{Chodosh-Li-Minter-Stryker,mazet}).
We say that $M$ satisfies $\biRic+c\geq 0$ if $\operatorname{biRic}(p)+c\geq 0$ for any point $p$ on $N$.
We can show the following.
 \begin{corollary}\label{corollary1}
     Let $(N,\partial N, g)$ be a smooth $(n+1)$-dimensional Riemannian manifold with boundary and $n\leq 4$. Assume that $\operatorname{biRic}+\frac{5-n}{4}H^2\geq 0$ and $\partial N$ is mean convex. Then, for any noncompact two-sided stable free boundary CMC immersion $(M,\partial M) \rightarrow (N,\partial N,g)$ with the mean curvature $H$, either
     \begin{itemize}
         \item[1,] $M$ splits isometrically as $\Sigma\times \mathbb{R}_{\geq 0}$, where $\operatorname{Ric}_\Sigma\geq 0.$ Moreover, the second fundamental form satisfies $A(\partial_t,\partial_t)=-\frac{n-3}{2}H$ where $\partial_t$ is the unit vector field in the direction $\mathbb{R}$ and $A(e_i,e_i)=\frac{1}{2}H$ where $e_i$ in the tangent space of $\Sigma$. The normal Ricci is constant, i.e., $\Ric^N(\nu,\nu)=-\frac{n^2-5n+8}{4}H^2.$
         \item[2,]  $M$ has no interior ends.
     \end{itemize}
 \end{corollary}
The stability in the above corollary means the nonnegativity of the quadratic form associated to the Jacobi operator with respect to any compactly supported test function. It usually refers to strong stability in the literature, for example, see \cite{Barbosa-Berard-twisted-eigenvalue}. This corollary in the minimal case can be compared to \cite[Theorem 5.3]{wuyujie}, where the author estimates the number of nonparabolic ends under stronger curvature conditions.

 In Section \ref{section2}, we prepare some calculations about the weighted length functional $L_u^\alpha$. In Section \ref{section3}, we show how to construct an $L_u^\alpha$-minimizing ray passing through any given point. In Section \ref{section4}, we prove a splitting theorem for manifolds with boundary and pointwise nonnegative Ricci curvature. In the last Section \ref{section5}, we prove our main result and its corollary.

 \subsection{Acknowledgements}
We thank the anonymous referee for his comments, which improve this paper. The first author is supported by NSFC No. 12401058 and the Talent
Fund of Beijing Jiaotong University No. 2024XKRC008. The second author is supported by the China Postdoctoral Science Foundation (No. 2024M751604).

\section{A weighted length functional}\label{section2}
Assume that $(M,g)$ is a noncompact manifold with boundary $\partial M$ and the unit outward normal $\eta$. Let $u$ be a positive $C^{2,\beta}$ function on $M$ for some $\beta\in(0,1)$. Fix a positive constant $\alpha$ which will be determined in later sections. For any curve $\gamma(s)$ parametrized by arclength $s$, consider the following functional
\begin{equation}\label{weighted-length-functional}
    L^\alpha_u(\gamma)=\int_\gamma u^\alpha ds.
\end{equation}

\begin{definition}
    A curve \(\gamma: I\subset \mathbb{R} \rightarrow M\) is said to be a \textit{\(L_u^\alpha\)-minimizing curve in the free boundary sense} if it satisfies the following two conditions for any \(s_1, s_2 \in I\):
    \begin{enumerate}
        \item For any curve \(\tilde{\gamma}\) connecting \(\gamma(s_1)\) and \(\gamma(s_2)\), we have
        \[
        L_u^\alpha(\tilde{\gamma}) \geq L_u^\alpha(\gamma([s_1, s_2])).
        \]
        \item For any two curves \(\tilde{\gamma}_1\) and \(\tilde{\gamma}_2\), each connecting \(\gamma(s_i)\) (\(i = 1, 2\)) to \(\partial M\), we have
        \[
        L_u^\alpha(\tilde{\gamma}_1) + L_u^\alpha(\tilde{\gamma}_2) \geq L_u^\alpha(\gamma([s_1, s_2])).
        \]
    \end{enumerate}
	In particular, we say $\gamma$ is \textit{$L_u^\alpha$-minimizing} if $\gamma$ satisfies only the first condition.
\end{definition}
\begin{remark}
    If $\gamma: I \rightarrow M$, parametrized by arc length, is an $L_u^\alpha$-minimizing curve in the free boundary sense and is a complete curve, then $I$ can only be $[0, +\infty)$ or $(-\infty, +\infty)$ after a reparametrization. 
    \begin{itemize}
        \item If $I = (-\infty, +\infty)$, we call $\gamma$ a \textit{$L_u^\alpha$-minimizing line in the free boundary sense}.
        \item If $I = [0, +\infty)$, we call $\gamma$ a \textit{$L_u^\alpha$-minimizing ray in the free boundary sense}.
    \end{itemize}
    In particular, $\gamma$ cannot be a closed segment or loop.
\end{remark}

\subsection{Weighted minimizing ray}
Let $\gamma(s):[0,b)\rightarrow M$ be a curve on $M$ with $\gamma(0)\in \partial M$ parametrized by arclength $s$. We say that it is proper if $\gamma\cap\partial M=\gamma(0)$. We allow that $b=+\infty$.

We denote by $\{e_1,\cdots, e_n\}$ a system of orthonormal vector fields along $\gamma$ such that $e_n=\frac{\partial}{\partial\gamma_s}$ and $\partial_se_i$ is tangential. This is the so-called relatively parallel frame (or Bishop frame, \cite{bishop1975Frame}). For simplicity, denote $e_n$ by $\partial s.$ Notice that $\partial s|_{\gamma(0)}=-\eta$ if $\gamma$ is perpendicular to $\partial M$ at $\gamma(0)$. The geodesic curvature vector of $\gamma$ is defined to be $\vec{\kappa}=-\nabla_{\partial s}\partial s$ where $\nabla$ is the Levi-Civita connection of $M.$

\begin{lemma}\label{first-variation}
    If $\gamma_t$ is a smooth 1-parameter family of proper curves with $\gamma_0=\gamma$, $\gamma_t(0)\in \partial M$ and $\gamma_t([a,b))=\gamma([a,b))$ for $a<b$. Then
    \begin{equation}\label{first-derivative of L}
        \left.\frac{d L^\alpha_u(\gamma_t)}{dt}\right|_{t=0}=\int_\gamma \langle\nabla u^\alpha+u^\alpha\vec{\kappa},(\left.\frac{\partial \gamma_t}{\partial t}\right|_{t=0})^{\perp}\rangle ds+\langle\frac{\partial\gamma_t}{\partial t},-\partial s\rangle|_{\gamma(0)}u^\alpha(\gamma(0)).
    \end{equation}
    In particular, the critical curve $\gamma$ of $L^\alpha_u$ satisfies
    \begin{align}\label{critical-equation of L}
        \kappa_i:=\langle \vec{\kappa},e_i\rangle=-\alpha\nabla_{e_i}\log u
    \end{align}
    along $\gamma$ for each $i=1,\cdots,n-1$ and it is perpendicular to $\partial M$ at $\gamma(0).$

    \end{lemma}

    By standard calculation, we obtain

    \begin{lemma}\label{second-variation of L in vector field}
     Let $\gamma_t$ be a smooth 1-parameter family of proper curves with $\gamma_0=\gamma$, $\gamma_t(0)\in\partial M$ and $\gamma_t([a,b))=\gamma([a,b))$ for $a<b$. Denote $V:=\frac{\partial \gamma_t}{\partial t}$. Then the second variation of $L_u^\alpha$ is
    \begin{align*}
        \frac{d^2 L^\alpha_u(\gamma_t)}{dt^2}&=\int_\gamma |\nabla^\perp_{\partial s}V|^2u^\alpha-R(\partial s,V,\partial s, V)u^\alpha+\operatorname{div}_\gamma(\nabla_V V)u^\alpha\\
        &+\int_\gamma 2\operatorname{div}_\gamma(V)\nabla_V u^\alpha+\nabla_V\langle\nabla u^\alpha, V\rangle
        \end{align*}
        where $R(\partial s,V,\partial s,V)$ is the sectional curvature of the plane spanned by  $\{\partial s,V\}$.
    \end{lemma}
    
    Using Lemma \ref{first-variation} and integration by parts, we obtain
\begin{lemma}\label{second-variation of L in one direction}
 Under assumptions of Lemma \ref{second-variation of L in vector field}. Suppose that $\gamma$ is an $L_u^\alpha$-minimizing curve in the free boundary sense and $V|_{t=0}=\phi e_i$. Then the corresponding quadratic form $Q_u(\phi):=\left.\frac{d^2 L^\alpha_u(\gamma_t)}{dt^2}\right|_{t=0}\geq 0$ where
    \begin{align*}
        Q_u(\phi)=&\int_{\gamma} \phi_s^2u^\alpha-R(\partial s,e_i,\partial s,e_i)u^\alpha \phi^2-(\alpha^2+\alpha)\phi^2u^\alpha(\nabla_{e_i}\log u)^2ds\\
        &+\int_\gamma \alpha\phi^2u^{\alpha-1}(\nabla^2u)(e_i,e_i)ds-h_{\partial M}(e_i,e_i)u^{\alpha}\phi^2\big|_{\gamma(0)}.
    \end{align*}
    Here, $h_{\partial M}(\cdot,\cdot)=\langle\nabla_\cdot\eta,\cdot\rangle$ is the second fundamental form of $\partial M$ with respect to the outer unit normal $\eta.$
\end{lemma}
\begin{proof}
The result follows from Lemma \ref{second-variation of L in vector field}. In fact, note that
\[|\nabla^\perp_{\partial s}V|=\phi_s;\]
\[\operatorname{div}_\gamma(V)=\langle \nabla_{\partial s}V,\partial s\rangle=\langle\vec{\kappa},V\rangle=-\alpha\phi\nabla_{e_i}\log u;\]
\[\nabla_V\langle \nabla u^\alpha,V\rangle=\alpha u^{\alpha-1}(\nabla^2 u)(V,V)+\alpha(\alpha-1)u^{\alpha-2}(\nabla_V u)^2+\langle \nabla u^\alpha,\nabla_V V\rangle.\]
On the other hand, integration by parts and the free boundary property yield
\begin{align*}
    \int_\gamma \operatorname{div}_\gamma(\nabla_VV)u^\alpha&=\int_\gamma \operatorname{div}_\gamma(\nabla^\perp_VV+\nabla^T_VV)u^\alpha\\
    &=\int_\gamma u^\alpha \langle\nabla_VV,\vec{\kappa}\rangle+\operatorname{div}_\gamma(u^\alpha\nabla^T_VV)-\langle\nabla^T u^\alpha,\nabla_VV\rangle\\
    &=\int_\gamma \langle -\nabla u^\alpha,\nabla_VV\rangle+\operatorname{div}_\gamma(u^\alpha\nabla^T_VV)\\
    &=u^{\alpha}\langle \nabla_VV,-\partial s\rangle|_{\gamma(0)}+\int_\gamma \langle -\nabla u^\alpha,\nabla_VV\rangle\\
&=-u^\alpha\phi^2h_{\partial M}(e_i,e_i)|_{\gamma(0)}+\int_\gamma \langle -\nabla u^\alpha,\nabla_VV\rangle.
\end{align*}
Plugging the above equations into Lemma \ref{second-variation of L in vector field} completes the proof.

\end{proof}

Summing over all $i$ we have
\begin{lemma}\label{summation of second variation}
    Let $\gamma$ be an $L_u^\alpha$-minimizing curve in the free boundary sense. Then for any compactly supported function $\phi$ on $\gamma$, we obtain
    \begin{align*}
        (n-1)\int_\gamma \phi_s^2u^\alpha ds &\geq \int_\gamma \left(\operatorname{Ric}(\partial s,\partial s)-\alpha\frac{\Delta u}{u}\right)u^\alpha\phi^2 ds-2\alpha\int_\gamma \phi u^{\alpha-1}u_s\phi_s ds\\&+\alpha\int_\gamma\phi^2u^\alpha|\nabla^\perp\log u|^2 ds-\alpha(\alpha-1)\int_\gamma u^{\alpha-2}\phi^2u_s^2\\
        &-\alpha\phi^2u^{\alpha-1}u_s\big|_{\gamma(0)}+H_{\partial M}u^{\alpha}\phi^2\big|_{\gamma(0)}.
    \end{align*}
\end{lemma}
\begin{proof}
    Directly add the result in Lemma \ref{second-variation of L in one direction} over $i$, then we have
    \begin{align*}
        (n-1)&\int_\gamma \phi_s^2u^\alpha ds \geq \int_\gamma \operatorname{Ric}(\partial s,\partial s)u^\alpha\phi^2 ds+\sum_{i=1}^{n-1}h_{\partial M}(e_i,e_i)u^{\alpha}\phi^2\big|_{\gamma(0)}\\
        &+\int_\gamma (\alpha^2+\alpha)\phi^2u^{\alpha}\sum_{i=1}^{n-1}(\nabla_{e_i}\log u)^2-\alpha\phi^2 u^{\alpha-1}\sum_{i=1}^{n-1}(\nabla^2 u)(e_i,e_i) ds.
    \end{align*}
    Note that
    \begin{align*}
        \Delta u&=\sum_{i=1}^{n-1}(\nabla^2 u)(e_i,e_i)+(\nabla^2 u)(\partial s,\partial s)\\
        &=\sum_{i=1}^{n-1}(\nabla^2 u)(e_i,e_i)+u_{ss}+\sum_{i=1}^{n-1}\kappa_i\nabla_{e_i}u\\
        &=\sum_{i=1}^{n-1}(\nabla^2 u)(e_i,e_i)+u_{ss}-\alpha u\sum_{i=1}^{n-1}(\nabla_{e_i}\log u)^2.
    \end{align*}
    Simplify the inequality and do an integration by parts to $u_{ss}$. This leads to the inequality in the lemma.
\end{proof}
As a corollary, we have
\begin{proposition}\label{keyinequality}
     Let $\gamma$ be an $L_u^\alpha$-minimizing curve in the free boundary sense. Then for any compactly supported function $\psi$ on $\gamma$, we obtain
     \begin{align}\label{the key inequality}
 c_1(n,\alpha,\varepsilon)\int_\gamma \psi^2_s ds&\geq  c_2(n,\alpha,\varepsilon)\int_\gamma u^{-2}\psi^2u_s^2 ds\nonumber\\
         &+\int_\gamma \left(\operatorname{Ric}(\partial s,\partial s)-\alpha\frac{\Delta u}{u}\right)\psi^2 ds+\alpha\int_\gamma\psi^2|\nabla^\perp\log u|^2 ds\\
         &+\psi^2\left.\left(H_{\partial M}-\alpha\frac{u_s}{u})\right)\right|_{\gamma(0)}\nonumber
     \end{align}
     where 
     \[c_1(n,\alpha,\varepsilon)= n-1+\frac{\alpha(n-3)}{4\varepsilon}>0\]
     and
     \[c_2(n,\alpha,\varepsilon)=\alpha(1-\frac{n-1}{4}\alpha-(n-3)\varepsilon)\]
     which is positive if $\alpha<\frac{4}{n-1}$ and $\varepsilon$ is small, or $n=3$ and $\alpha< \frac{4}{n-1}$.
\end{proposition}
\begin{proof}
    Plug  $\phi=u^{-\alpha/2}\psi$ into Lemma \ref{summation of second variation} gives
    \begin{align*}
         (n-1)\int_\gamma \psi^2_s ds&\geq  \alpha(1-\frac{n-1}{4}\alpha)\int_\gamma u^{-2}\psi^2u_s^2 ds+\alpha(n-3)\int_\gamma u^{-1}\psi u_s\psi_s ds\\
         &+\int_\gamma \left(\operatorname{Ric}(\partial s,\partial s)-\alpha\frac{\Delta u}{u}\right)\psi^2 ds+\alpha\int_\gamma\psi^2|\nabla^\perp\log u|^2 ds\\
&+\psi^2\left.\left(H_{\partial M}-
         \alpha\frac{u_s}{u})\right)\right|_{\gamma(0)}.
     \end{align*}
     Using the Cauchy-Schwarz inequality, we have
     \[\int_\gamma u^{-1}\psi u_s\psi_s ds\leq \varepsilon\int_\gamma u^{-2}\psi^2u_s^2ds+\frac{1}{4\varepsilon}\int_\gamma \psi_s^2 ds.\]
     Plugging it into the above inequality completes the proof.
\end{proof}
\vskip.2cm

\subsection{Weighted minimizing line} In this part, we adapt some results from the previous subsection to curves entirely contained in $M \setminus \partial M$. Let $\gamma(s): (-b, b) \rightarrow M$ be a curve on $M$ parametrized by arc length $s$. We allow $b = +\infty$.

Using calculations similar to those above, we obtain the following two results.

\begin{lemma}
    If $\gamma_t$ is a smooth 1-parameter family of proper curves with $\gamma_0=\gamma$, $\gamma_t\cap \partial M=\emptyset$ and $\gamma_t((-b,-a])\cup \gamma_t([a,b)])=\gamma((-b,-a])\cup \gamma([a,b)])$ for $0<a<b$.  Then
    \begin{equation}\label{first-derivative of L-line}
        \left.\frac{d L^\alpha_u(\gamma_t)}{dt}\right|_{t=0}=\int_\gamma \langle\nabla u^\alpha+u^\alpha\vec{\kappa},(\left.\frac{\partial \gamma_t}{\partial t}\right|_{t=0})^{\perp}\rangle ds.
    \end{equation}
    In particular, the critical curve $\gamma$ of $L^\alpha_u$ satisfies
    \begin{align}\label{critical-equation of L-line}
        \kappa_i:=\langle\vec{\kappa},e_i\rangle=-\alpha\nabla_{e_i}\log u
    \end{align}
    along $\gamma$.
    
\end{lemma}

\begin{proposition}\label{keyinequality-line}
     Let $\gamma$ be an $L_u^\alpha$-minimizing curve in the free boundary sense. Then for any compact supported function $\psi$ on $\gamma$, we obtain
     \begin{align}\label{the key inequality-line}
 c_1(n,\alpha,\varepsilon)\int_\gamma \psi^2_s ds&\geq  c_2(n,\alpha,\varepsilon)\int_\gamma u^{-2}\psi^2u_s^2 ds\nonumber\\
         &+\int_\gamma \left(\operatorname{Ric}(\partial s,\partial s)-\alpha\frac{\Delta u}{u}\right)\psi^2 ds+\alpha\int_\gamma\psi^2|\nabla^\perp\log u|^2 ds
     \end{align}
     where $c_1(n,\alpha,\varepsilon)$ and $c_2(n,\alpha,\varepsilon)$ are the same as the constants in Proposition \ref{keyinequality}.
\end{proposition}

\subsection{Rigidity of $L_u^\alpha$-minimizing curve in the free boundary sense}
When the function $u$ and the curvatures of $M$ satisfy certain conditions, we can derive a rigidity result. Specifically, we prove the following.

\begin{proposition}\label{rigidity proposition}
   Let $(M,g)$ be a Riemannian manifold with boundary. Suppose that $\alpha<\frac{4}{n-1}$ and $\gamma$ is a complete $L_u^\alpha$-minimizing curve in the free boundary sense in $M$. If 
    \[\operatorname{Ric}-\alpha u^{-1}\Delta u\geq 0\]
    and 
    \[ (H_{\partial M}-\alpha u^{-1}u_s)|_{\partial M}\geq 0.\]
    Then $|\nabla u|=0$  and $\operatorname{Ric}-\alpha u^{-1}\Delta u=0$ along $\gamma$. 
\end{proposition}
\begin{proof}
    The proof directly follows from Proposition \ref{keyinequality} and Proposition \ref{keyinequality-line}. In fact, if $\gamma$ is a complete ray, we let 
    \[\psi=\begin{cases}
        1, &\ s\in[0,R]\\
        \frac{2R-s}{R}, &\ s\in[R,2R]\\
        0, & \ \text{otherwise}.
    \end{cases}\]
    Plug it into \eqref{the key inequality}, we obtain
    \[\frac{c_1}{R}\geq c_2\int_0^Ru^{-2}u_s^2 ds+\int_0^R \operatorname{Ric}(\partial s,\partial s)-\alpha\frac{\Delta u}{u}ds+\alpha\int_0^R|\nabla^\perp\log u|^2 ds.\]
    By letting $R\rightarrow \infty,$ we have that
    $u_s=0$, $|\nabla^\perp u|=0$ and $\operatorname{Ric}-\alpha u^{-1}\Delta u=0$ along $\gamma.$ For the other case when $\gamma$ is a complete line, we have the same conclusion by plugging the following $\psi$ into \eqref{the key inequality-line}
      \[\psi=\begin{cases}
        1, &\ s\in[-R,R]\\
        \frac{2R-s}{R}, &\ |s|\in[R,2R]\\
        0, & \ \text{otherwise}.
    \end{cases}\]
    The proof is complete.
\end{proof}

\section{Construction of Weighted Minimizing Rays and Lines }\label{section3}

Suppose $M$ has nonnegative $\alpha$-Ricci curvature in the spectral sense. Then there exists a positive function $u \in C^{2,\beta}(M)$ satisfying \eqref{equation 1 of u} and \eqref{equation 2 of u}.
Note that throughout this section, we also assume the mean convexity of the boundary $\partial M$.

The main theorem of this section is
\begin{theorem}\label{theorem of passing}
	Suppose $M$ has nonnegative $\alpha$-Ricci curvature in the spectral sense and mean convex boundary. If $M$ contains a complete $L_u^\alpha$-minimizing curve $\gamma$ in the free boundary sense, then for any $p\in M$, there exists a complete $L_u^\alpha$-minimizing curve in the free boundary sense, denoted by $\gamma_p$, that passes through $p$.
	\label{thm_existence_of_line}
\end{theorem}

The proof of the above theorem is inspired by Liu \cite{liugang} and Chodosh-Eichmair-Moraru \cite{chodoshsplittingtheorem}. A similar result has been proved by Chu-Lee-Zhu \cite{zhu1}. Notice that in our case, the convergence of a sequence of $L_u^\alpha$ minimizing curves in the free boundary sense may produce a complete $L_u^\alpha$ minimizing line that is captured.

We first prove the following lemma.

\begin{lemma}
    \label{lem_upper_bound}
    Suppose $\gamma$ is an $L_u^\alpha$-minimizing curve in the free boundary sense in a compact set $K \subset M$. Then, the total length of $\gamma$ is uniformly bounded, and the weighted length of $\gamma$ is uniformly bounded by a constant $C$, where $C$ depends on the diameter of $K$, $\alpha$, and the infimum and supremum of $u$ on $K$.
\end{lemma}

\begin{proof}
    We establish the following inequality:
    \[
    \inf_{K} u^\alpha \int_{\gamma} d\gamma \leq \int_{\gamma} u^\alpha d\gamma \leq \int_{\tilde{\gamma}} u^\alpha d\tilde{\gamma} \leq \sup_{\tilde{K}} u^\alpha \int_{\tilde{\gamma}} d\tilde{\gamma} \leq \sup_{\tilde{K}} u^\alpha (\mathrm{diam}(K) + \varepsilon),
    \]
    where $\tilde{\gamma}$ is a curve connecting the endpoints of $\gamma$, $\tilde{K}=\{x\in M:d(x,K)\le \mathrm{diam}(M)\}$, and the total length of $\tilde{\gamma}$ is bounded by $\mathrm{diam}(K) + \varepsilon$.
    Here, we have used the fact that $\tilde{\gamma}\subset \tilde{K}$. Taking $\varepsilon \rightarrow 0^+$, we obtain the desired result.
\end{proof}

We will need the following lemma to construct a perturbed weighted function $u_{r,t}$ near an existing $L_u^\alpha$-minimizing curve in the free boundary sense.

\begin{lemma}
	\label{lem_existenceU}
	For any $p_0 \in M$, there exist positive constants $r_0=r_0(n,M, p_0)$ and $t_0=t_0(n,M,p_0)$ such that for any $r\le r_0$ and any $p$ in the closure of $B_{2r}(p_0)$, there exists a positive function $u_{r,t}$ on $M$ for any $0<t\le t_0$ satisfying the following conditions:
\begin{enumerate}
	\item $u_{r,t}\rightarrow u$ as $r,t\rightarrow 0^+$ in the $C^2$ sense, and $u_{r,t}\rightarrow u$ smoothly for any fixed $r>0$ as $t\to 0^+$.
	\item $u_{r,t}=u$ on $M\backslash B_{3r}(p)$, and $u_{r,t}<u$ in $B_{3r}(p)$.
	\item $-\alpha \Delta u_{r,t}+\operatorname{Ric}\cdot u_{r,t}> 0$ in $B_{3r}(p)\backslash B_r(p)$.
	\item $\frac{\partial u_{r,t}}{\partial \eta}\ge 0$ on $\partial M$.
\end{enumerate}
\end{lemma}
\begin{proof}
	This is a purely local construction.
Such perturbations have been widely used in many works (e.g., \cite{ehrlich1976metric,liugang, Carlotto-Chodosh-Eichmair-PMT}), and we will essentially follow the approach presented in \cite[Appendix J]{Carlotto-Chodosh-Eichmair-PMT}.
	First, we consider the case $p_0 \in \partial M$.
	We choose the standard Fermi coordinate near $p_0$ such that $p_0=0$ and $M$ is locally given by $\{x_1\le 0\}$ near $0$, and the metric $g_{ij}$ at $0$ is $\delta_{ij}$.
	In particular, we can write $\eta=\frac{\partial  }{\partial x_1}$.

	Given any $p \in B_{2r_0}(p_0)$ for $r_0$ small enough, we choose the function $d_p$ by
	\[
		d_p(\cdot)=\mathrm{dist}_{\delta}(p,\cdot),
	\]
	where $\mathrm{dist}_\delta$ denotes the distance function with respect to the metric $\delta$ (which is the Euclidean metric using the Fermi coordinate).
	A straightforward calculation shows that
	\begin{equation}
		\frac{\partial d_p^2}{\partial \eta}= 0.
		\label{eq:pfGradDp}
	\end{equation}
    Note that
	\begin{align*}
		\Delta_g d_p^2 ={}&
		\sum_{i,j=1 }^{n}g^{ij}\frac{\partial^2d_p^2}{\partial x_i\partial x_j} - \sum_{i,k=1 }^{n}\Gamma_{ii}^k \frac{\partial d_p^2}{\partial x_k}\\
		\le{}&  2n + \sum_{i,j=1 }^{n}2n\sup_{B_{5r_0}(p_0)}|g^{ij}-\delta| + 2d_p\sup_{B_{5r_0}(p_0)}|\Gamma_{ii}^j|,
	\end{align*}
where $g^{ij}$ denotes the inverse of $g_{ij}$, and $\Gamma_{ij}^k$ denotes the Christoffel symbols of $g$.
	When we choose $r_0$ small enough (not depending on $p$), both $d_p$ and $|g^{ij}-\delta|$ are small enough in $B_{5r_0}(p_0)$ and hence, we can obtain
	\begin{equation}
		\frac{1}{2}\le |\nabla^g d_p|_g\le 2,\quad \Delta_gd_p^2\le 2n+\frac{1}{2} \text{ in }B_{5r_0}(p_0).
	\end{equation}
	Note that this implies
	\[
		d_p\Delta_g d_p\le n.
	\]

Next, we choose a function $f(s)$ defined on $\mathbb{R}$ such that
\begin{enumerate}
	\item $f(s)$ is a negative constant on $(-\infty,\frac{1}{2}]$,
	\item $f(s)=-\exp(k/(s-3))$ on $[1,3)$,
	\item $f(s)=0$ when $s\ge 3$.
\end{enumerate}

Then, for any $l>0$ and $s \in [1,3)$, we have
\[
	sf''(s)+l f'(s)=-\exp(\frac{k}{s-3})\frac{k}{(s-3)^4}\left(ks+2s(s-3)-l(s-3)^2\right).
\]
If we choose $k>4+4l$, then
\begin{equation}
	sf''(s)+lf'(s)<0\quad \text{ for }s\in [1,3).
	\label{eq:pfProperyf}
\end{equation}
Note that $f'(s)>0$ and $ f''(s)<0$ for $s\in [1,3)$.

We consider
\[
	v_r=r^4 f(\frac{d_p}{r}).
\]
Then, we know
\[
	v_r<0 \text{ in }B_{3r}(p),
\]
and define
\[
	u_{r,t}=(1+tv_r)u.
\]
In particular, for $t_0$ small enough, we know $u_{r,t}>0$ for any $0<r\le r_0$ and $t\in (0,t_0]$.
In $B_{3r}(p)\backslash B_r(p)$, we compute
\begin{align*}
	\alpha\Delta u_{r,t}={}&\alpha(1+tv_r)\Delta u+\alpha tu\Delta v+2\alpha t\nabla u\cdot \nabla v\\
	={}& \mathrm{Ric} \cdot u_{r,t}+\alpha t (u \Delta v+2\nabla u\cdot \nabla v)\\
	\le{}& \mathrm{Ric} \cdot u_{r,t}+\alpha t \left( u r^2 f''(\frac{d_p}{r})|\nabla d_p|_g^2+ur^3f'(\frac{d_p}{r})\Delta d_p\right. \\
    &+\left.2|\nabla u|_g r^3f'(\frac{d_p}{r})|\nabla d_p|_g\right)\\
	\le{}& \mathrm{Ric} \cdot u_{r,t}+\frac{\alpha t u r^3}{d_p} \left( \frac{d_p}{4r} f''(\frac{d_p}{r})+nf'(\frac{d_p}{r})\right.\\
    &\left.+ \frac{12\||\nabla u|_g\|_{L^\infty(B_{5r_0}(p_0))}r_0}{\inf _{B_{5r_0}(p_0)}u}f'(\frac{d_p}{r})\right).
\end{align*}
Now, we select $l=4n+\frac{48\||\nabla u|_g\|_{L^\infty(B_{5r_0}(p_0))}r_0}{\inf _{B_{5r_0}(p_0)}u}$ and $k=4l+4+1$. It follows from \eqref{eq:pfProperyf} that
\[
	\alpha \Delta u_{r,t}< \mathrm{Ric} \cdot u_{r,t}
\]
in $B_{3r}(p)\backslash B_r(p)$.
Now, we compute
\[
	\frac{\partial u_{r,t}}{\partial \eta}=(1+tv_r)\frac{\partial u}{\partial \eta}+tu \frac{\partial v_r}{\partial \eta}=tu r^3 f'(\frac{d_p}{r}) \frac{\partial d_p}{\partial \eta}\ge 0
\]
by \eqref{eq:pfGradDp}.
It is easy to verify that other conditions are also satisfied by the construction of $u_{r,t}$.

If $p_0 \in M\backslash \partial M$, we can choose $r_0$ sufficiently small such that $5r_0<\mathrm{dist}_g(p_0,\partial M)$, and then we do not need to consider the boundary condition.
The construction of $u_{r,t}$ is the same as above, except that we use the distance function in normal coordinates rather than in Fermi coordinates.
\end{proof}

\begin{theorem}
	\label{prop_existence}
	Let $\gamma:I\rightarrow M$ be a complete $L_u^\alpha$-minimizing curve in the free boundary sense in $M$.
	Then, for any $s \in I$, there exists a positive constant $r_0=r_0(s,\gamma)<\frac{1}{3}$ such that for any $p \in M\setminus \gamma$ with $\mathrm{dist}_g(p,\gamma)=\mathrm{dist}_g(p,\gamma(s))=2r$ and $0<r\le r_0$, we can construct another complete $L_u^\alpha$-minimizing curve in the free boundary sense, denoted by $\gamma_r$, which passes through the closure of $B_r(p)$.
\end{theorem}

\begin{proof}
First, we consider the case that $I=[0,+\infty)$.
We fix any $s \in [0,+\infty)$ and let $p_0=\gamma(s)$.
Let $r_0,t_0$ be the positive constants given by Lemma \ref{lem_existenceU} for $p_0$.
Let $p$ be the point such that $\mathrm{dist}_g(p,p_0)=\mathrm{dist}_g(p,\gamma)=2r$ and $u_{r,t}$ be the function constructed in Lemma \ref{lem_existenceU} for $p$ and $t\in (0,t_0)$.

We construct new smooth weighted functions $u_{r,t,h}$ for any $h> 0$ by requiring
\begin{enumerate}
	\item $u_{r,t,h}=u_{r,t}$ in $\left\{ q \in M:\mathrm{dist}_g(q,\gamma([0,h+s+1]))\le 2h \right\}$,
	\item 
    $u_{r,t,h}\ge u_{r,t}$ in $M\backslash \left\{ q \in M:\mathrm{dist}_g(q,\gamma([0,h+s+1]))\le 2h\right\}$,
    
	\item $u_{r,t,h}\ge \max\{ 1,u_{r,t} \}$ in $ \left\{ q \in M:\mathrm{dist}_g(q,\gamma([0,h+s+1]))\geq 2h +1\right\}$.
\end{enumerate}
Let $\gamma_{r,t,h}$ to be the curve from $\gamma(h+s+1)$ to $\partial M$ which minimizes the functional $L_{u_{r,t,h}}^\alpha$.
Note that since $M$ is complete and $u_{r,t,h}\ge 1$ outside a compact set, the curve $\gamma_{r,t,h}$ is well-defined and has finite length by Lemma \ref{lem_upper_bound}.
It is straightforward to verify that $\gamma_{r,t,h}$ is an $L_{u_{r,t,h}}^\alpha$-minimizing curve in the free boundary sense.

We claim that $\gamma_{r,t,h}$ has a non-empty intersection with $B_{3r-\varepsilon}(p)$
for $\varepsilon>0$ small enough, which does not depend on $h$.
Note that by the construction of $u_{r,t}$, we have
\[
	u_{r,t}^\alpha\ge (1-C\varepsilon)u^\alpha
\]
in $B_{3r}(p)\backslash B_{3r-\varepsilon}(p)$ for some positive constant $C$.

If $\gamma_{r,t,h}\cap B_{3r-\varepsilon}(p)=\emptyset$, then we have
\begin{align}\label{eq:pfcomparison}
	0< {}&\int_0^{s+h+1} u^\alpha d\gamma-\int_0^{s+h+1} u_{r,t}^\alpha d\gamma \\
    ={}&\int_0^{s+h+1} u^\alpha d\gamma-\int_0^{s+h+1} u_{r,t,h}^\alpha d\gamma\nonumber \\
	\le{}& \int_0^{s+h+1} u^\alpha d\gamma-\int_{ \gamma_{r,t,h}} u_{r,t,h}^\alpha d\gamma_{r,t,h}\nonumber \\
	\le{}& \int_0^{s+h+1} u^\alpha d\gamma-\int_{ \gamma_{r,t,h}\cap B_{3r}(p)} (1-C\varepsilon)u^\alpha d\gamma_{r,t,h}-\int_{ \gamma_{r,t,h}\backslash B_{3r}(p)} u^\alpha d\gamma_{r,t,h}\nonumber \\
={}& \int_0^{s+h+1} u^\alpha d\gamma-\int_{\gamma_{r,t,h}} u^\alpha d\gamma_{r,t,h}+C\varepsilon \int_{ \gamma_{r,t,h}\cap B_{3r}(p)} u^\alpha d\gamma_{r,t,h}\nonumber \\
	\le{}& \int_0^{s+h+1} u^\alpha d\gamma-\int_0^{s+h+1} u^\alpha d\gamma+C\varepsilon \int_{ \gamma_{r,t,h}\cap B_{3r}(p)} u^\alpha d\gamma_{r,t,h}\nonumber \\
	={}& C\varepsilon \int_{ \gamma_{r,t,h}\cap B_{3r}(p)} u^\alpha d\gamma_{r,t,h}.	\nonumber
\end{align}
Now, we show that the integral in the last line is independent of \( h \).  
We claim that the total length of \( \gamma_{r,t,h} \cap B_{3r}(p) \) is uniformly bounded by a constant independent of \( h \).

To prove this, denote by \( \{\hat{\gamma}_i\}_{i=1}^N \) the collection of connected components of \( \gamma_{r,t,h} \cap B_{4r}(p) \) that pass through \( B_{3r}(p) \).  
Since \( u_{r,t} = u_{r,t,h} \) inside \( B_{4r}(p) \), each \( \hat{\gamma}_i \) is an \( L_{u_{r,t}}^\alpha \)-minimizing curve in the free boundary sense.

Let \( \Lambda = \sup_{B_{4r}(p)} u_{r,t} \) and \( \lambda = \inf_{B_{4r}(p)} u_{r,t}^\alpha \). There exists a constant \( \varepsilon_0 > 0 \) such that for any \( i \neq j \),  
\[
\mathrm{dist}_g(\hat{\gamma}_i \cap B_{3r}(p), \hat{\gamma}_j \cap B_{3r}(p)) > \varepsilon_0.
\]  
Suppose, for contradiction, that this fails. Without loss of generality, assume  
\[
\mathrm{dist}_g(\hat{\gamma}_1 \cap B_{3r}(p), \hat{\gamma}_2 \cap B_{3r}(p)) \leq \varepsilon_0.
\]  
Let \( \hat{\gamma}_1 = \gamma_{r,t,h}((t_{10}, t_{11})) \), \( \hat{\gamma}_2 = \gamma_{r,t,h}((t_{20}, t_{21})) \) with \( t_{10} < t_{11} \le t_{20} < t_{21} \).
Choose \( \tau_1 \in (t_{10}, t_{11}) \), \( \tau_2 \in (t_{20}, t_{21}) \) such that \( \gamma_{r,t,h}(\tau_i) \in \overline{B_{3r}(p)} \) for \( i = 1, 2 \) and  
\[
\mathrm{dist}_g(\gamma_{r,t,h}(\tau_1), \gamma_{r,t,h}(\tau_2)) \leq \varepsilon_0.
\]  
Construct a new curve \( \tilde{\gamma} \) by replacing \( \gamma_{r,t,h}((\tau_1, \tau_2)) \) with a geodesic \( \hat{\gamma} \) connecting \( \gamma_{r,t,h}(\tau_1) \) and \( \gamma_{r,t,h}(\tau_2) \). Then,  
\[
\int_{\hat{\gamma}} u_{r,t}^\alpha \, d\hat{\gamma} \leq \Lambda\int_{\hat{\gamma}} d\hat{\gamma} \leq \Lambda \varepsilon_0.
\]  
On the other hand, since \( u_{r,t}^\alpha \geq \lambda> 0 \) on \( B_{4r}(p) \), we have  
\[
\int_{\gamma_{r,t,h}((\tau_1, \tau_2))} d\gamma_{r,t,h} \leq \frac{1}{\lambda} \int_{\gamma_{r,t,h}((\tau_1, \tau_2))} u_{r,t}^\alpha \, d\gamma_{r,t,h}.
\]  
Note that the length of \( \gamma_{r,t,h}((\tau_1, \tau_2)) \) is at least $2r$ since it will touch the boundary of \( B_{4r}(p) \) somewhere.
Hence,  
\[
\int_{\hat{\gamma}} u_{r,t}^\alpha \, d\hat{\gamma} \leq \Lambda \varepsilon_0 \leq \frac{\Lambda \varepsilon_0}{2r\lambda} \int_{\gamma_{r,t,h}((\tau_1, \tau_2))} u_{r,t}^\alpha \, d\gamma_{r,t,h}.
\]  
By choosing \( \varepsilon_0 \) sufficiently small such that \( \frac{\Lambda \varepsilon_0}{2r\lambda} < 1 \), we obtain a contradiction to the \( L_{u_{r,t}}^\alpha \)-minimality of \( \gamma_{r,t,h} \), since \( \tilde{\gamma} \) would have strictly smaller weighted length.

Therefore, the number \( N \) of such components is uniformly bounded, independent of \( h \). Moreover, by Lemma \ref{lem_upper_bound}, the length of each \( \hat{\gamma}_i \) is uniformly bounded by a constant independent of \( h \). Hence, the total length of \( \gamma_{r,t,h} \cap B_{3r}(p) \) is uniformly bounded independent of \( h \). It follows that the integral in the last line of \eqref{eq:pfcomparison} is also uniformly bounded by a constant independent of \( h \).
Then, we can choose $\varepsilon$ small enough to obtain a contradiction, since the right-hand side of the first line of \eqref{eq:pfcomparison} is a constant that does not depend on $h$.

For any compact $K\subset M$, the total length of each component of $\gamma_{r,t,h}\cap K$ is uniformly bounded by Lemma \ref{lem_upper_bound}.
The critical equation \eqref{critical-equation of L} also implies the geodesic curvature of $\gamma_{r,t,h}$ is uniformly bounded in $K$.
Then, for fixed $r,t>0$, we can choose a subsequence of $\gamma_{r,t,h}$ that converges locally and smoothly to $\tilde{\gamma}_{r,t}$ as $h \rightarrow +\infty$.
Note that $\tilde{\gamma}_{r,t}$ may have multiple connected components.
Since $u_{r,t,h}\rightarrow u_{r,t}$ in any compact set, each component of $\tilde{\gamma}_{r,t}$ is a complete $L_{u_{r,t}}^\alpha$-minimizing curve in the free boundary sense.

Moreover, $\tilde{\gamma}_{r,t}$ has non-empty intersection with the closure of $B_{3r-\varepsilon}(p)$.
We define $\gamma_{r,t}$ to be the component of $\tilde{\gamma}_{r,t}$ which passes through the closure of $B_{3r-\varepsilon}(p)$.

\vskip.1cm 
 \textbf{Claim:} $\gamma_{r,t}\cap B_{r}(p)\neq \emptyset$.
 \vskip.1cm
  
This follows from Proposition \ref{rigidity proposition}. Lemma \ref{lem_existenceU} ensures that $u_{r,t}$ satisfies the assumption of Proposition \ref{rigidity proposition}.  Thus $\operatorname{Ric}-\alpha u_{r,t}^{-1}\Delta u_{r,t}=0$ along $\gamma$, contradicting (iii) of Lemma \ref{lem_existenceU}.

Now, we may further choose $t\rightarrow 0^+$ and, up to a subsequence, have $\gamma_{r,t}$ converging locally and smoothly to $\tilde{\gamma}_r$. We choose $\gamma_r$ to be the component of $\tilde{\gamma}_r$ that passes through the closure of $B_{r}(p)$.
Then, $\gamma_r$ is a complete $L_u^\alpha$-minimizing curve in the free boundary sense and passes through the closure of $B_{r}(p)$. We remark that it is not known whether $\gamma_r$ is a ray or a line.
See Figure \ref{fig:minimizing-in-free-boundary-sense}.
\begin{figure}[ht]
    \centering
    \begingroup
\def\svgwidth{0.8\columnwidth}
	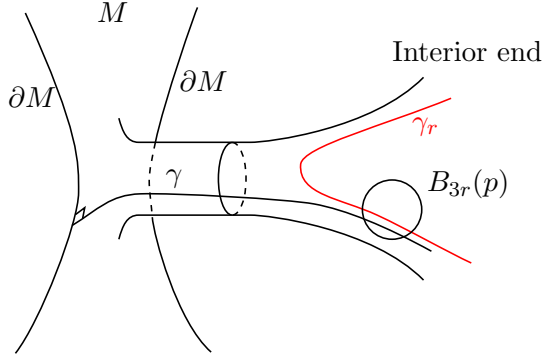
	\endgroup

    \caption{The case where $\gamma_r$ is a line}
    \label{fig:minimizing-in-free-boundary-sense}
\end{figure}

\vskip.2cm
Now, we consider the case that $I=(-\infty,+\infty)$.
For any $p_0 \in \gamma$, we assume $\gamma(0)=p_0$ after a reparametrization.
The choice of $u_{r,t}$, $u_{r,t,h}$ is the same as above.
The difference lies in the construction of $\gamma_{r,t,h}$.
For a fixed $h>0$, we define the class of curves $A_h$ as follows.
A curve $\beta$ belongs to $ A_h$ if $\beta$ satisfies one of the following conditions:
\begin{enumerate}
    \item $\beta$ is a curve connecting $\gamma(-h-1)$ and $\gamma(h+1)$,
    \item $\beta$ has two components $\beta_1$ and $\beta_2$, where each $\beta_i$ is a curve connecting $\gamma((-1)^i(h+1))$ and $\partial M$.
\end{enumerate}
We define $\gamma_{r,t,h}$ to be the curve in $A_h$ that minimizes $L_{u_{r,t,h}}^\alpha$.
Such a curve exists because $M$ is complete and $u_{r,t,h}\ge 1$ outside a compact set.
The computation in \eqref{eq:pfcomparison} still holds, and we can choose $\varepsilon$ small enough to ensure that $\gamma_{r,t,h}$ has non-empty intersection with $B_{3r-\varepsilon}(p)$.
Then, we can repeat the argument from the case $I=[0,+\infty)$ to obtain a complete $L_u^\alpha$-minimizing curve in the free boundary sense, denoted by $\gamma_r$, which passes through the closure of $B_r(p)$.
\end{proof}

Now, we are ready to prove the main theorem stated at the beginning of this section.

\begin{proof}[Proof of Theorem \ref{theorem of passing}]
	We define the set $\Omega\subset M$ by
	\begin{align*}
        \Omega:={}&\{p\in M:\ \text{there exists a complete} \\
		  &\text{ $L_u^\alpha$-minimizing curve in the free boundary sense through}\ p\}.
	\end{align*}
	It is clear that $\Omega$ is a non-empty closed subset of $M$ by the assumption of the theorem and the compactness of $L_u^\alpha$-minimizing curves in the free boundary sense.

	If $\Omega\neq M$, we pick any $p_1$ in $M\backslash \Omega$.
	Let $p_0 \in \Omega$ be the nearest point to $p_1$, and let $r_0$ be the positive constant given by Proposition \ref{prop_existence} for $p_0$.
        We choose $\gamma$ to be any $L_u^\alpha$-minimizing curve in the free boundary sense through $p_0$.
	Let $\beta$ be a shortest curve connecting $p_0$ and $p_1$, with $\beta(0)=p_0$ parametrized by arclength.
	It is easy to see that $B_{2r}(\beta(2r))\cap \Omega=\emptyset$ for any $r>0$ sufficiently small.
        It implies $$\mathrm{dist}_g(\beta(2r),\gamma)=\mathrm{dist}_g(\beta(2r),p_0)=2r.$$
        Then, we can apply Theorem \ref{prop_existence} for $r<r_0$ to conclude that $\overline{B_r(\beta(2r))}\cap \Omega\neq \emptyset$.
        This is a contradiction.
	Hence, $\Omega=M$.
\end{proof}

\section{Splitting theorem with interior end}\label{section4}

We need the following proposition, which can be proved by a standard argument. Recall that an interior end of a noncompact manifold $M$ with boundary is an unbounded component of $M\setminus B_R$ that has empty intersection with $\partial M.$

\begin{proposition}\label{existence of geodesic ray}
    Let $(M, g)$ be a smooth noncompact Riemannian manifold with boundary. If $M$ has an interior end, then there exists an $L_u^\alpha$-minimizing curve in the free boundary sense in $M$.
\end{proposition}

\begin{proof}
    This part is similar to the construction in the proof of Theorem \ref{prop_existence}. Fix a point $p \in \partial M$. Since $M$ has an interior end by assumption, we can choose a sufficiently large $R$ such that $M \setminus B_R(p)$ has at least one unbounded component, denoted by $E$, that does not intersect $\partial M$.

    Now, consider a family of exhaustions $B_{R+n}(p)$ of $M$ for $n = 1, 2, \dots$. We also choose a perturbation of the weighted function $u_n$ such that:
    \begin{itemize}
        \item $u_n = u$ in $B_{R+n}(p)$,
        \item $u_n \geq u$ on $M$,
        \item $u_n > \max\{1, u\}$ in $M \setminus B_{R+n+1}(p)$.
    \end{itemize}
    Since $\partial B_{R+n}(p) \cap E$ is compact, there exists a curve $\gamma_n$ connecting $\partial B_{R+n}(p) \cap E$ and $\partial M$ that minimizes the weighted length $L_{u_n}^\alpha$. Note that $\gamma_n \cap B_R(p) \neq \emptyset$ for all $n$.

    Taking $n \to \infty$ and passing to a subsequence, $\gamma_n$ converges locally and smoothly to $\tilde{\gamma}$, which passes through the closure of $B_R(p)$.
	Each component of $\tilde{\gamma}$ is $L_u^\alpha$-minimizing in the free boundary sense.
	We then select one of the connected components of $\tilde{\gamma}$ that passes through the closure of $B_R(p)$. This is the desired curve.
\end{proof}

\begin{remark}
	Note that if we choose $u \equiv 1$, then each $\gamma_n$ is the shortest geodesic connecting $\partial B_{R+n}(p) \cap E$ and $\partial M$, and the limit $\gamma$ is a minimizing geodesic ray starting from $\partial M$ since $M$ is complete.
\end{remark}

\begin{theorem}\label{theorem with pointwise ric>0}
 Let $(M,g)$ be a smooth noncompact Riemannian manifold with mean convex boundary. Assume that the Ricci curvature of $M$ is nonnegative and $M$ has an interior end. Then, $M$ is isometric to $\Sigma\times \mathbb{R}_{\geq 0}$, where $\Sigma$ is a closed manifold with nonnegative Ricci curvature.
\end{theorem}
If we replace the existence of an interior end with the assumption that the boundary of $M$ is compact, this result is due to Kasue \cite{kasue} and  Croke-Kleiner \cite{croke-bruce}. The proof essentially follows from Kasue \cite{kasue}. The interesting point is that the existence of an interior end forces the boundary of the manifold to be compact and the number of interior ends to be exactly one.
\begin{proof}
 By choosing $u\equiv 1$,
Proposition \ref{existence of geodesic ray} provides us with a free boundary minimizing geodesic ray $\gamma:[0,\infty)\rightarrow M$ such that $\operatorname{dist}(\gamma(t),\partial M)=t$ for any $t\geq 0$. Then result follows from \cite[Theorem 1.8]{sakurai2017}.
\end{proof}

\vskip.2cm
\section{Proof of the main theorem}\label{section5}

In this section, we provide a proof of the main theorem.

\begin{proof}[Proof of Theorem \ref{maintheorem}]
    First, we claim that if $M$ has an interior end, it must split. By the assumption, let $u$ satisfy 
    $$-\alpha \Delta u+\operatorname{Ric}\cdot u=0 \ \text{on}\ M,$$
 \[\frac{\partial u}{\partial \eta}=0 \ \text{on}\ \partial M.\]
Recall the weighted length function 
 \[L_u^\alpha(\gamma)=\int_\gamma u^\alpha.\]
By Proposition \ref{existence of geodesic ray}, we can find a complete curve $\tau$ that is $L_u^\alpha$-minimizing in the free boundary sense.

 With this curve as a starting point, by Theorem \ref{theorem of passing}, for any point on $M$ there exists a complete $L_u^\alpha$ minimizing curve passing through it. It follows from Proposition \ref{rigidity proposition} that $|\nabla u|=0$ on $M$. Thus, $u$ is a positive constant, so $\operatorname{Ric}\geq 0$ on $M$, i.e., the Ricci curvature is pointwise nonnegative. Hence, the result follows from Theorem \ref{theorem with pointwise ric>0}.
\end{proof}

\begin{proof}[Proof of Corollary \ref{corollary1}]
    The stability inequality is
    \[\int_M|\nabla \varphi|^2-(\operatorname{Ric}^N(\nu,\nu)+|A|^2)\varphi^2-\int_{\partial M}h_{\partial N}(\nu,\nu)\varphi^2\geq 0.\]
    Let $\{e_i\}_{i=1}^n$ be a local orthonormal frame of $M$. We assume that $\operatorname{Ric}=\operatorname{Ric}(e_1,e_1)$. The Gauss equation implies
    \begin{align*}
        \int_M|\nabla \varphi|^2+&\operatorname{Ric}\varphi^2-\int_{\partial M}h_{\partial N}(\nu,\nu)\varphi^2\\
        &\geq \int_M (\operatorname{Ric}^N(\nu,\nu)+\sum_{i=1}^n R^N_{1i1i}+|A|^2-h_{11}H+\sum_{i=1}^nh_{1i}^2)\varphi^2\\
        &\geq \int_M \left(\operatorname{biRic}(e_1,\nu)+\frac{5-n}{4}H^2\right)\varphi^2.
    \end{align*}
   In the above second inequality, we have used the following inequality (see \cite{chengxustructure}):
    \[|A|^2+h_{11} H-\sum_{i=1}^nh_{1i}^2\geq \frac{5-n}{4}H^2.\]
    We remark that the equality holds in above inequality if and only if $h_{11}=-\frac{n-3}{2}H,\ \ h_{ii}=\frac{1}{2}H, \ \text{and}\ h_{ij}=0, \ \forall\ i\neq j.$
    
    It follows from  assumptions of the theorem that for any compactly supported function $\varphi\in C_c(M)$, it holds that
    \[\int_M|\nabla \varphi|^2+\operatorname{Ric}\varphi^2-\int_{\partial M}h_{\partial N}(\nu,\nu)\varphi^2\geq 0.\]
    Then there exists a positive function $u$ on $M$ satisfying
    $$-\Delta u+\operatorname{Ric}\cdot u=0 \ \text{on}\ M,$$
 \[\frac{\partial u}{\partial \eta}=h_{\partial N}(\nu,\nu)u \ \text{on}\ \partial M.\]

 Using this $u$ in Proposition \ref{keyinequality} and choose $\alpha=1$ (this needs $n\leq 4$), the boundary term of the inequality in Proposition \ref{keyinequality} now becomes
 \[\psi^2\left.\left(H_{\partial M}-\alpha\frac{u_s}{u}\right)\right|_{\gamma(0)}=\left.\psi^2(H_{\partial M}+h_{\partial N}(\nu,\nu))\right|_{\gamma(0)}.\]
 Since $M$ is a free boundary hypersurface in $(N,\partial N)$, then the above term is nonnegative if we assume the mean convexity of $\partial N$. Then all the analysis in the proof of Theorem \ref{maintheorem} follows if $n\leq 4$. That is, if we assume the existence of an interior end, then $M$ is isometric to $\Sigma\times \mathbb{R}_{\geq0}$ where $\Sigma$ is a closed manifold.

Denote $e_1:=\partial_t$ the unit vector field in the direction of $\mathbb{R}_{\geq 0}$. Let $r$ be the distance function to the boundary $\partial M$ and let $\phi(r)$ be a Lipschitz function that is $1$ when $r<R$, and is $0$ when $r>2R$. Moreover, $|\varphi_r|\leq 1/R$. Then, by the stability inequality
\begin{align*}
    \frac{1}{R^2}\int_{\{r\leq 2R\}}1 &\geq \int_M|\nabla \phi|^2\\
    &\geq \int_{\{r\leq 2R\}}\Ric^N(\nu,\nu)+|A|^2\\
    &= \int_{\{r\leq 2R\}}|A|^2+\biRic(e_1,\nu)-\sum_{i=1}^nR^N_{i1i1}\\
    & = \int_{\{r\leq 2R\}}|A|^2+\biRic(e_1,\nu)-\Ric^M(e_1,e_1)+ H h_{11}-\sum_{i=1}^nh_{1i}^2\\
    &=\int_{\{r\leq 2R\}}\biRic(e_1,\nu)+|A|^2+ H h_{11}-\sum_{i=1}^nh_{1i}^2\\
    &\geq \int_{\{r\leq 2R\}} \biRic+\frac{5-n}{4}H^2.
\end{align*}
In the second equality, we have used the Gauss equation. In the third equality, we have used the fact that $\Ric^M(e_1,e_1)=0$. 

By the assumption,  letting $R$ tend to infinity, we obtain equalities in the above. Then
\[\biRic(e_1,\nu)+\frac{5-n}{4}H^2=0.\]
\[h_{11}=-\frac{n-3}{2}H,\ \ h_{ii}=\frac{1}{2}H, \ \text{and}\ h_{ij}=0, \ \forall\ i\neq j.\]
Then
\begin{align*}
    \Ric^N(\nu,\nu)&=\biRic(e_1,\nu)+Hh_{11}-\sum_{i=1}^nh_{1i}^2\\
    &=\frac{-n^2+5n-8}{4}H^2.
\end{align*}

\end{proof}

\bibliographystyle{alpha}
\bibliography{mybib1}     

\end{document}